\documentclass[12pt,a4paper]{amsart}
\usepackage{amsfonts}
\usepackage{amsthm}
\usepackage{amssymb}
\usepackage{amsmath}
\usepackage{mathabx}
\usepackage{tikz-cd}

\def\bbC{{\mathbb C}}
\def\bbD{{\mathbb D}}

\def\bbT{{\mathbb T}}

\def\cD{{\mathcal D}}

\def\cH{{\mathcal H}}

\def\cR{{\mathcal R}}

\def\vphi{{\varphi}}
\def\om{{\omega}}
\def\whd{{\widehat{\cD}}}
\def\wcd{{\widecheck{\cD}}}
\def\whom{{\widehat{\omega}}}
\def\asym{{\asymp}}
\def\les{{\lesssim}}
\def\gtr{{\gtrsim}}
\def\dak{{D(a_k,r)}}
\def\awp{{A_\omega^p}}
\def\bzw{{b_z^\om}}

\usepackage[colorlinks=blue]{hyperref}
\usepackage{graphics}
\usepackage{hyperref}
\hypersetup{
colorlinks=true,
linkcolor=red,
citecolor=cyan,
}
\newtheorem{remark}{Remark}

\newtheorem{definition}{Definition}[section]
\newtheorem{theorem}{Theorem}[section]

\newtheorem{lem}{Lemma}[section]

\newtheorem{lemm}[theorem]{Lemma}
\newtheorem{lemma}{Lemma}

\newtheorem{theo}{Theorem}

\title[Summing Carleson Embdeddings]{Summing Carleson measure on radial weighted Begman spaces}
\author{Mingjin Li, Jianren Long$^*$}
\thanks{*Corresponding author. }
\date{}
\address{Mingjin Li \newline
	School of Mathematical Sciences, Guizhou  Normal University, Guiyang, 550025,  China.}
\email{limingjin2022@163.com }
\address{Jianren Long \newline
	School of Mathematical Sciences, Guizhou  Normal University, Guiyang, 550025,  China.}
\email{longjianren2004@163.com }
\begin{document}

\maketitle
\begin{abstract}
    For \(\mu\) is a positive Borel measure on \(\bbD\), The \(r\) summing Carleson embdedding \(J_\mu: A_w^p\to L^q(\mu)\) are characterized in this paper, some conditions which ensure that the Carleson embedding for \(J_\mu: A_w^p\to L^q(\mu)\) is \(r\) summing are obtained for all \(0<p,q<\infty.\)
\end{abstract}  

\section{Introduction}
Let $\bbD=\{z\in\bbC:|z|<1\}$ be the unit disc and $\bbT=\{z\in\bbC:|z|=1\}=\partial\bbD$ be its boundary. $\cH(\bbD)$ denotes the holomorphic functions on $\bbD$. For $0<p<\infty$, let 
\begin{align*}
& M_p(r,f)=\left(\frac{1}{2\pi}\int_0^{2\pi}|f(re^{i\theta})|^pd\theta\right)^{\frac{1}{p}}, \\
&M_\infty(r,f)=\sup_{|z|=r}|f(z)|.
\end{align*} 
Hardy space $H^p$ consists of all $f\in\cH(\bbD)$ such that $\|f\|_{H^p}=\sup\limits_{0<r<1} M_p(r,f)<\infty$, in particularly, if $p=\infty$, $H^\infty$ is the bounded analytic function spaces with norm $\|f\|_{H^\infty}=\sup\limits_{z\in\bbD}|f(z)|<\infty.$
 A weight \(\om\) is called radial if \(\om(z)=\om(|z|)\), for $0<p<\infty$, the radial weighted Bergman space $A_\om^p$ consists of functions $f\in\cH(\bbD)$ such that 
$$\|f\|_{A_\om^p}=\left(\int_\bbD |f(z)|^p\om(z)dA(z)\right)^\frac{1}{p}<\infty,$$ where $dA(z)=\frac{1}{\pi}dxdy=\frac{r}{\pi}drd\theta$ is the normalized Lebesgue measure on \(\bbD\). If \(\om(z)=(1-|z|^2)^\alpha,\alpha>-1\), then \(A_\om^p\) is the standard radial weight Bergman spaces, we write \(A_\alpha^p\) for simple. If  \(\om(z)=1\), this is the standard Bergman space \(A^p\).

We say that a radial weight \(\om\) belongs to \(\whd\) if $\whom(z)=\int_{|z|}^1\om(s)ds$ satisfies the doubling conditions, i.e. there exists $C=C(\om)\geq 1$ such that $$\whom(r)\leq C\whom(\frac{1+r}{2}), \quad 0\leq r<1.  $$
Furthermore, if there exist $K=K(\omega)>1$ and $C=C(\omega)>1$ such that \[\whom(r)\ge C\whom\left(1-\frac{1-r}{K}\right),\quad 0\le r<1,\]
then we write $\omega\in\wcd$. The definitions of both $\whd$ and $\wcd$ have their obvious geometric interpretations. The intersection $\whd\cap\wcd$ is denoted by $\cD$.  The class \(\cD\) naturally arises in operator theory on spaces of analytic functions, \(\om(z)=(1+\alpha)(1-|z|^2)^\alpha,\alpha>-1\) is the classical \(\cD\) function. For \(\om\in\cD\) belong to \(\cR\) if $$\whom(r)\asymp\whom(r)(1-r), 0\leq r<1.$$
The class \(\cD\) is the largest possible class of functions for which most of the results concerning the standard Bergman space \(A_\alpha^p\) holds. For more property of radial weights, it can be found in \cite{Pelaez2014,Pelaez2016summer,Pelaez2018} and therein references. In this paper, we shall use the following two properties of \(\cD\) class.
\begin{lemma}\cite{Pelaez2018}\label{hatpropety}
    Let \(\omega\) be a radial weight. Then the following conditions are equivalent:

\begin{enumerate}
    \item[(i)] \(\omega \in \widehat{\mathcal{D}}\);
    
    \item[(ii)] There exist \(C = C(\omega) > 0\) and \(\beta = \beta(\omega) > 0\) such that
    \[
    \widehat{\omega}(r) \leq C \left(\frac{1 - r}{1 - t}\right)^{\beta} \widehat{\omega}(t), \quad 0 \leq r \leq t < 1;
    \]
    
    \item[(iii)] There exist \(C = C(\omega) > 0\) and \(\gamma = \gamma(\omega) > 0\) such that
    \[
    \int_{0}^{t} \left(\frac{1 - t}{1 - s}\right)^{\gamma} \omega(s) \, \mathrm{d}s \leq C \widehat{\omega}(t), \quad 0 \leq t < 1;
    \]
    
    
    \item[(v)] \(\frac{1}{\om(S(z))} \asymp \widehat{\omega}(z)(1 - |z|)\), \(|z| \to 1^{-}\);
    
    \item[(vi)] There exists \(\gamma = \gamma(\omega) \geq 0\) such that
    \[
    \int_{\mathbb{D}} \frac{\omega(z)}{|1 - \overline{\zeta}z|^{\gamma + 1}} \, \mathrm{d}A(z) \asymp \frac{\widehat{\omega}(\zeta)}{(1 - |\zeta|)^{\gamma}}, \quad \zeta \in \mathbb{D}.
    \]
    
\end{enumerate}

\end{lemma}

\begin{lemma}\cite{Pelaez2018}\label{chekproperty}
    Let \(\omega\) be a radial weight. Then the following statements are equivalent:

\begin{enumerate}
    \item[(i)] \(\om\in\wcd\):
    
    \item[(ii)] There exist \(C = C(\omega) > 0\) and \(\beta = \beta(\omega) > 0\) such that
    \[
    \widehat{\omega}(t) \leq C \left(\frac{1 - t}{1 - r}\right)^{\beta} \widehat{\omega}(r), \quad 0 \leq r \leq t < 1.
    \]
    
\end{enumerate}
\end{lemma}

Now we back to our main subject. Given a  positive Borel measure \(\mu\) on \(\bbD\), $p,q>0$, we consider the Carleson embedding
\begin{align*}
    J_\mu:A_\om^p&\hookrightarrow L^q(\mu)\\
     f &\longmapsto f
\end{align*}
One of the most famous result such that embedding is bounded on Hardy spaces was introduced in \cite{Carleson1962},
let \(W(\zeta,h)=\{z\in\bbD:1-h\leq|z|,\,|\arg(z\overline{\zeta})|\leq h\}\), Hasting \cite{Hastings1975} characterized  the boundedness of the Carleson embedding for the case of Bergman spaces: \(J_\mu\) is bounded if and only if \[\sup_{\zeta\in\partial\bbD}=\mu(W(\zeta,h))=O(h^2).\] 

 As mentioned above, Recently, a series of works \cite{Pelaez2013,Pelaez2014,Pelaezprojection2016,Pelaez2016summer,Pelaez2018,Pelaez2021} by Pel\'aze et al.  have extended the research on operator theory and harmonic analysis in traditional Bergman spaces to a larger class of Bergman spaces with radial weight. Recently, Pel\'aez  et al. \cite{Pelaez2015} characterized the boundedness of Carleson embedding for all $0<p,q<\infty$ and \(\om\in\cD\) by using the weighted maximal function \[M_{\om,\alpha}(\mu)(z)=\sup_{z\in S(a)}\frac{\mu(S(a))}{(\om(S(a)))^\alpha},\quad z\in\bbD.\] 
 Let non-tangential approach regions \[\Gamma(\zeta)=\left\{z\in\bbD:|\theta-\arg z|<\frac{1}{2}\left(1-\frac{|z|}{r}\right)\right\}, \,\,\zeta=re^{i\theta}\in\bbD\setminus \{0\}\] and the related tents 
 \[T(z)=\{\zeta\in\bbD: z\in \Gamma(\zeta)\},\] the characterization of the Carleson embedding are as follows.
\begin{theo}\cite{Pelaez2015}\label{cm}
    Let \( 0 < p, q < \infty \), \(\omega \in \widehat{D}\) and \(\mu\) be a positive Borel measure on \(\mathbb{D}\).
    If \( p > q \), the following conditions are equivalent.
\begin{enumerate}

    \item \(\mu\) is a \(q\)-Carleson measure for \(A_\omega^p\);  
    \item The function  
    
    \[
    B_\mu(z) = \int_{\Gamma(z)} \frac{d\mu(\zeta)}{\omega(T(\zeta))}, \quad z \in \mathbb{D}\backslash\{0\},
    \]
    
    belongs to \(L_\omega^{\frac{p}{p-q}}\);  
   \item \(M_\omega(\mu) \in L_\omega^{\frac{p}{p-q}}\).
\end{enumerate}
If \(p\leq q\), the following conditions are equivalent.
\begin{enumerate}
    \item \(\mu\) is a \(q\)-Carleson measure for \(A_\omega^p\);
    \item \( M_{\omega,q/p}(\mu) \in L^\infty \);
\item  \( z \mapsto \frac{\mu(\Delta(z,r))}{(\omega(S(z)))^{\frac{q}{p}}} \) belongs to \( L^\infty \) for any fixed \( r \in (0,1) \).
\end{enumerate}

\end{theo}

 Once boundedness is established, it naturally leads to a series of operator theory issues, such as compactness, Schatten classes, summability and so on.

The main purpose of this paper is the characterization of finite positive Borel measure \(\mu\) on \(\bbD\) such that \(J_\mu\) is \(r-\) summing. Before this procedure, let us recall that the definitions of summing operators.
\begin{definition}

      Let \( r \geq 1 \) and \( T : X \to Y \) be a bounded linear operator between Banach spaces. We say that \( T \) is \( r \)-summing if there is a constant \( C \geq 0 \) such that regardless of the natural number \( n \) and regardless of the choice of \( x_1, x_2, \ldots, x_n \) in \( X \), we have

\[
\left( \sum_{k=1}^n \|Tx_k\|_Y^r \right)^{1/r} \leq C \sup_{x^* \in B_{X^*}} \left( \sum_{k=1}^n |x^*(x_k)|^r \right)^{1/r} = C \sup_{(a_k) \in B_{\ell^{r'}}} \left\| \sum_{k=1}^n a_k x_k \right\|_X
\]

where \( X^* \) denotes the dual space of \( X \), and \( B_{X^*} \) denotes the unit ball of the Banach space \( X^* \). The best constant \( C \) for which the inequality always holds is denoted by \( \pi_r(T) \).
\end{definition}

Throughout this paper, \(\Pi_r(X,Y)\) denote the set of all \(r-\) summing operators from \(X\) into \(Y\). If \(r=1\), we say that \(\Pi_1(X, Y\) is the class of absolutely summing operators and \(\pi_r(T)\) is the \(r-\) summing norm, it is the best constant for \(C\geq 0\). The class of \(r-suming\) operators forms an operator ideal.\cite{Diestel1995}  It is also not difficult check that \(\Pi_r(X,Y)\subset B(X,Y)\), the space of all bounded operators from \(X\) to \(Y\). It is well known that \(\Pi_r(X,Y)\) is a Banach space with norm \(\pi_r\).  \(\Pi_{r}(X,Y)\) is an operator ideal between Banach spaces: for any \(T\in\Pi_{r}(X,Y)\), and for any two Banach spaces \(X_{0}\), \(Y_{0}\) such that \(S\in B(X_{0},X)\) and \(U\in B(Y,Y_{0})\), we have \(UTS\in\Pi_{r}(X_{0},Y_{0})\) with \[\pi_r(UTS)\leq\|U\|\pi_r(T)\|S\|.\]
The class of spaces \(\Pi_{r}(X,Y)\) where \(r\geq 1\) is monotone. That is, for any \(1\leq r\leq s<\infty\), we have \(\Pi_{r}(X,Y)\subset\Pi_{s}(X,Y)\) and the relationship \(\pi_{s}(T)\leq\pi_{r}(T)\), for any \(T\in\Pi_{r}(X,Y)\).

we say that a Banach space \( X \) has cotype \( q \geq 2 \), if there is a constant \( C \geq 0 \) such that no matter how we select finitely many vectors \( x_1, x_2, \ldots, x_n \) from \( X \), we have

\[
\left( \sum_{k=1}^n \|x_k\|^q \right)^{1/q} \leq C \left( \int_0^1 \left\| \sum_{k=1}^n r_k(t)x_k \right\|^2 dt \right)^{1/2},
\]

where \( (r_k) \) is a Rademacher sequence.

For any measure space \( (\Omega, \Sigma, \mu) \), the Lebesgue type space \( L^q(\Omega, \mu) \) has cotype \(\max\{q, 2\}\). Then for any Banach spaces \( X \) and \( Y \), we have the following properties:

\begin{enumerate}
    \item If \( X \) has cotype 2, then \(\Pi_2(X, Y) = \Pi_1(X, Y)\).
    \item If \( X \) has cotype \( 2 < q < \infty \), then \(\Pi_1(X, Y) = \Pi_r(X, Y)\) for all \( 1 < r < q'\).
    \item If \( X \) and \( Y \) both have cotype 2, then \(\Pi_r(X, Y) = \Pi_1(X, Y)\) for every \( 1 < r < \infty\).
\end{enumerate}

In the 1950s, Grothendieck first used the aforementioned notation in his research on summability. one of his the most famous result is that every continuous linear operator from \( l^1 \) to \( l^2 \) is an absolutely summing operator. However, this achievement did not gain attention until the 1960s. Later, it received widespread recognition in the fields of operator theory and the geometric properties of Banach spaces.  Pietsch established some fundamental properties of summing operators and conducted related research on certain Lebesgue spaces. Meanwhile, Lindenstrauss and Pelczy\'nski emphasized \cite{Lindenstrauss1968} how these properties could helpful in studying the geometric properties of Banach spaces. Subsequently, more systematic monographs emerged \cite{Diestel1995,Li2018,Wojtaszczyk1991}. It is worth noting that the ideal properties of some other operators are closely related to those of absolutely summing operators, such as \(r-\)integral operators, \(r-\)nuclear operators and so on. 

    Very recently, Lef\`evre et al. \cite{Lefevre2018} obtained a complete characterization of \(r-\) summing Carleson embedding on \(H^p\) when \(p> 1\). He et al. \cite{He2024} gave a complete characterization of \(r-\) summing Carleson embedding on \(A_\alpha^p\) for all \(p>1\).  However, most results in operator theory on Bergman spaces hold for functions of class \(\cD\), such as the boundedness, compactness, and Schatten class of operators. Therefore, a natural question arises: Does the related conclusions on summability also hold for functions of class \(\cD\)?
    
we will solve this problem for all \(p,q>1\). Before give our main results, we will need more notations. We reminder that the hyperbolic metric on \(\bbD\) is defined by 
\[\beta(z,w)=\frac{1}{2}\log\frac{1+\rho(z,w)}{1-\rho(z,w)}\] and \[\rho(z,w)\] denotes the pseudo-hyperbolic, defined by 
\[\rho(z,w)=|\vphi_w(z)|=\bigg|\frac{z-w}{1-z\overline{w}}\bigg|,\] let \(\Delta(z,r)=\{w:\rho(w,z)<r\}\) denotes the pseudo-hyperbolic disc and \(D(z,r)=\{w:\beta(w,z)<r\}\). By an argument of \cite{Duren2004,Zhu2007}, \(|\Delta(z,r)|\asymp |D(z,r)|\),  where \( |E| \) denotes the Lebesgue measure of \( E \subset T \). 
\begin{definition}\label{lattice}
    A sequence \(\{a_k\}\) in \(D\) is called a \(t\)-lattice (\(t > 0\)) in the Bergman metric if the following conditions are satisfied:

\begin{enumerate}
    \item \(\bbD=\bigcup\limits_{k=1}^\infty D(a_k, t)\).
    \item \(\beta(a_i, a_j) \geq \frac{t}{2}\) for all \(i\) and \(j\) with \(i \neq j\).
    \end{enumerate}
\end{definition}
By \cite{Duren2004,Zhu2007}, For any \(t>0\), there exists an \(t-\)lattice in the Bergman metric and pseudo-hyperbolic metric, and there exists a positive constant \(N\) such that ever point \(z\in\bbD\) belong to at most \(N\) disks \(D(a_k,t)\).  
\begin{lemma}\cite{Zhu2007}\label{rlattice}
 Suppose \( 0 < t < 1 \) and \(\{a_k\}\) is a \( t \)-lattice in the Bergman metric. For each \( k \), there exists a measurable set \( D_k \) with the following properties:
\begin{enumerate}
    \item \( D(a_k, t/4) \subset D_k \subset D(a_k, t) \) for all \( k \geq 1 \).
    \item \( D_j \cap D_i = \emptyset \) for \( i \neq j \).
    \item \( D_1 \cup D_2 \cup \cdots = \mathbb{D} \).
\end{enumerate}   
\end{lemma}
\(D_k\) arise in above is called regular lattice.  
The Carleson square \( S(I) \) based on an interval \( I \) on the boundary \( \bbT \) of \( \bbD \) is the set
\[
S(I) = \{re^{it} \in D : e^{it} \in I, \, 1 - |I| \leq r < 1\},
\]
 We associate to each \( a \in \bbD \setminus \{0\} \) the interval
\[
I_a = \{e^{i\theta} : |\arg(ae^{-i\theta})| \leq \frac{1-|a|}{2}\},
\]
and denote
\(S(a) = S(I_a).\) 
Throughout this paper, we shall use several times the following well known fact about summing operators.
\begin{lemma}\cite{Diestel1995}\label{propsumming}
    For any Banach spaces \(X\) and \(Y\), \(T:X\to Y\) is \(r\)-summing (\(r\geq 1\)) if and only if there is a constant \(C>0\) such that for any measurable space \((\Omega,\Sigma,\nu)\) and any continuous function \(F:\Omega\to X\), we have
\[
\left(\int_{\Omega}\|T\circ F\|^{r}d\nu\right)^{1/r}\leq C\sup_{s\in B_{X^{*}}} \left(\int_{\Omega}|s(F)|^{r}d\nu\right)^{1/r}.
\]
Moreover, the best \(C\) is \(\pi_{r}(T)\).
 \end{lemma}

Here and from now on, For each \( 1 < p < \infty \), we write \( p' \) for its conjugate exponent, that is,
\[
\frac{1}{p} + \frac{1}{p'} = 1.
\] $\om_x=\int_0^1r^x\om(r)\,dr$ for all $1\le x<\infty$. Further, as usual, we write $A(x)\lesssim B(x)$ or $B(x)\gtrsim A(x)$ for all $x$ in some set $I$ if there exists a constant $C>0$ such that $A(x)\le C B(x)$ for all $x\in I$. Further, if $A(x)\lesssim B(x)\lesssim A(x)$ in $I$, we say that $A(x)$ and $B(x)$ are comparable and write $A(x)\asymp B(x)$ for all $x\in I$, or simply $A\asymp B$.  

\section{Order bounded operators}
 Order bounded operators play a significant in the summing operators, we recall that \(J_\mu:A_\om^p\to L^q(\mu)\) is order bounded if and only if there exist \(h\in L^q(\mu)\) such that for every \(f\in B_{A_\om^p}\), we have \(|f|\leq h\) a.e. on \(\bbD\). We need the following results for order bounded operators. 
 \begin{lem}\label{order1}
    Let \( X \) be a Banach space and \(\mu\) be any Borel measure on \(\mathbb{D}\). Let \( T : X \to L^p(\mu) \), \( 1 \leq p < \infty \), be an order bounded operator. Then \( T \) is \( p \)-summing with

\[
\pi_p(T) \leq \| \sup_{f \in B_X} |T(f)| \|.
\]
 \end{lem}

\begin{lem}\label{order2}
 Let \(\om\in \cD\), \(0<p\leq q<\infty\), \(\mu\) be a positive finite Borel measure on \(\bbD\) and \(p\geq 1\).   Then \(J_\mu \) is order bounded if and only if 
 \[\int_\bbD  \frac{1}{\left(\whom{(z)}(1-|z|)\right)^{\frac{q}{p}}}d\mu(z)<\infty.\] 
\end{lem}
\begin{proof}
    Sice dilated function \(f_r=(frz)\) approximate the function \(f\) in \(\awp\) for \(\om\in\cD\), this implies that Polynomials are dense in \(\awp\) and hence \(\awp\) is separable, it follows that \(J_mu\) is order bounded if and only if  
    \[\int_\bbD\sup_{f\in B_{\awp}}|f(z)|^qd\mu(z)<\infty,\] 
    which is equivalent to 
    \[\int_\bbD\|L_z\|^q<\infty,\] where \(L_z\) is the point evaluation at \(z\in\bbD\). 
For \(\om\in\cD\), \(r\in (0,1)\) and \(0<p\leq q<\infty\), by \cite[p. 8]{Pelaez2018}
\begin{align*}
    |f(z)|^q&=(|f(z)|^{p})^{\frac{q}{p}}\\
    &\leq\left(\frac{1}{\left(\frac{r(1-|z|^2)}{4(1+r)}\right)^2}\int_{B(z,\frac{r(1-|z|^2)}{4(1+r)})}|f(w)|^pdA(w)\right)^{\frac{q}{p}}\\
    &\les\left(\frac{1}{\frac{\whom{(z)}}{1-|z|}(1-|z|)^2}\int_{D(z,r)}|f(w)|^p\frac{\whom{(w)}}{1-|w|}dA(w)\right)^\frac{q}{p}\\
    &\les \frac{1}{\left(\whom{(z)}(1-|z|)\right)^{\frac{q}{p}}}\|f\|_{\awp}^q
\end{align*}
\end{proof}
           \section{Summing multiplication operators}
Let For \( p \geq 1 \), the multiplier operator \( {M}_{\beta} \) on \( \ell^p \) is defined by

\begin{equation}\label{multi}
    {M}_{\beta}(e_n) = \beta_n e_n, \quad n = 1, 2, \cdots, 
\end{equation}

where \(\{e_n\}_{n \geq 1}\) denotes the canonical basis of \( \ell^p \) and \(\beta = (\beta_1, \beta_2, \ldots)\) is a sequence of complex numbers.

\begin{lem}\cite[Proposition 4.1]{Lefevre2018}\label{multiplier}
    Define \( M_\beta \) as in \eqref{multi}. Then with constants depending only on \( p \) and \( r \), we have:

\begin{enumerate}
    \item For \( 1 \leq p \leq 2 \) and every \( r \geq 1 \),
    \[
    \pi_r(\mathcal{M}_\beta) \approx \|\beta\|_{\ell^2} 
    \]

    \item For \( p \geq 2 \) and \( r \leq p' \),
    \[
    \pi_r(\mathcal{M}_\beta) \approx \|\beta\|_{\ell^{p'}} 
    \]

    \item For \( p \geq 2 \) and \( p' \leq r \leq p \),
    \[
    \pi_r(\mathcal{M}_\beta) = \|\beta\|_{\ell^{r}}
    \]

    \item For \( p \geq 2 \) and \( r \geq p \),
    \[
    \pi_r(\mathcal{M}_\beta) \approx \|\beta\|_{\ell^{p}} \cdot \
    \]
\end{enumerate}
\end{lem}

\subsection*{Atomic decomposition}

For \( 0 < p \leq \infty, 0 < q < \infty \) and a radial weight \(\omega\), the weighted mixed norm space \( A_{\omega}^{p,q} \) consists of \( f \in H(\mathbb{D}) \) such that

\[
\|f\|_{A_{\omega}^{p,q}}^{q} = \int_{0}^{1} M_{p}^{q}(r, f)\omega(r) \, dr < \infty.
\]

If \( q = p \), then \( A_{\omega}^{p,q} \) coincides with the Bergman space \( A_{\omega}^{p} \) induced by the weight \(\omega\).
for each \( K \in \mathbb{N} \setminus \{1\}, j \in \mathbb{N} \cup \{0\} \) and \( l = 0, 1, \ldots, K^{j+3} - 1 \), define the dyadic polar rectangle as

\[
Q_{j,l} = \left\{ z \in \mathbb{D} : r_j \leq |z| < r_{j+1}, \arg z \in \left[ 2\pi \frac{l}{K^{j+3}}, 2\pi \frac{l+1}{K^{j+3}} \right] \right\},
\]

where \( r_j = r_j(K) = 1 - K^{-j} \) as before, and denote its center by \( \zeta_{j,l} \). For each \( M \in \mathbb{N} \) and \( k = 1, \ldots, M^2 \), the rectangle \( Q_{j,l}^k \) is defined as the result of dividing \( Q_{j,l} \) into \( M^2 \) pairwise disjoint rectangles of equal Euclidean area, and the centers of these squares are denoted by \( \zeta_{j,l}^k \), respectively.

\begin{lem}\cite[Theorem 2]{Pelaze2021atmoic}
    Let \( 0 < p \leq \infty, \, 0 < q < \infty, \, K \in \mathbb{N} \setminus \{1\} \) and \(\omega \in D\). Then there exists \( M = M(p, q, \omega) > 0 \) such that \( A_{\omega}^{p, q} \) consists of functions of the form

\[
f(z) = \sum_{j,l,k} \lambda(f)^k_{j,l} \frac{(1 - |\xi^k_{j,l}|^2)^{M - \frac{1}{p}} \hat{\omega}(r_j)^{-\frac{1}{q}}}{(1 - \overline{\xi}^k_{j,l} z)^M}, \quad z \in \mathbb{D},
\]

where \(\lambda(f) = \{\lambda(f)^k_{j,l}\} \in \ell^{p,q}\) and

\[
\left\| \{\lambda(f)^k_{j,l}\} \right\|_{\ell^{p,q}} \asymp \| f \|_{A_{\omega}^{p,q}}.
\]
\end{lem}

Let \(f_j(z)=\frac{(1 - |\xi^k_{j,l}|^2)^{M - \frac{1}{p}} \hat{\omega}(r_j)^{-\frac{1}{p}}}{(1 - \overline{\xi}^k_{j,l} z)^M}\), then for \(\om\in\cR\), and \(z\in D_k\) and by an argument of \cite[Proposition4.1, 4.2]{He2024} we have
\[|f_j|^p\asymp \frac{1}{\whom(r_j)(1-r_j)}\asymp\frac{1}{\om(D_k)}.\] 

\begin{lem}\label{decomposition}
    Let \(p\geq 1\) and \(\mu\) be a Carleson measure on \(\awp\). Then there exists an operator \(S: l^p\to L^p(\mu)\) is \(r\)-summing if and only if for some sequence \(\beta\) such that \(M_\beta\) is \(r\)-summing on \(L^p(\mu)\).
\end{lem}
\begin{proof}
    Considering the following maps
    \[ l^p \xrightarrow{M_\beta} l^p \xrightarrow{T} L^p(\mu),
   \]
For \(a=\{a_k\}\in l^p\), let 
\[T(a)=\sum_ja_jf_j(z)\chi_{Q_j}(z),\] then \(T\) is bounded from \(l^p\) to \(L^p\). Indeed
\begin{align*}
    \|T(a)\|_{L^p(\mu)}&=\left(\int_\bbD\bigg|\sum_ja_jf_j\chi_{Q_j}\bigg|^pd\mu(z)\right)^\frac{1}{p}\\
    &=\left(\sum_j\int_{Q_j}|a_j|^p|f_j|^pd\mu(z)\right)^\frac{1}{p}\\
    &\asymp\left(\frac{\mu(Q_j)}{(\whom(a_j)(1-|a_j|))}\right)^\frac{1}{p}\left(\sum_j|a_j|^p\right)^\frac{1}{p}.
\end{align*}
Correcting the constants in the above process, we redefine the operator T as 
\begin{equation}\label{eqT}
    T(a)=\sum_ja_j\left(\frac{\whom(a_j)(1-|a_j|)}{\mu(Q_j)}\right)^\frac{1}{p}f_j(z)\chi_{Q_j}(z), 
\end{equation}
then 
\[\|T(a)\|_{L^p(\mu)}\asymp\|a\|_{l^p}.\]
Let
\[\beta=\left\{\left(\frac{\whom(a_j)(1-|a_j|)}{\mu(Q_j)}\right)^\frac{1}{p}\right\}_j\]
and 
\[S=M_\beta\circ T: a=\{a_k\}\in l^p\longmapsto f=\sum_j a_j f_j\chi_{Q_j}\in L^p(\mu),\] 
by the ideal property of \(\pi_r\) yields 
\[\pi_r(S)\leq\|T\|\pi_r(M_\beta)\les\pi_r(M_\beta),\] thus the sufficientness is proved.
Conversely, let \(X=\overline{L(f_j\chi_{Q_j})}\) be a closed subspace of \(L^p(\mu)\)
let
    \[U: a=\{a_j\}\in l^p\longmapsto f=\sum_j a_j\beta_j f_j\chi_{Q_j}\in X,\] 
    then \(U\) is bounded from \(l^p\) to \(X\) with \(\|U(a)\|_X\asymp \|a\|_{l^p}\), Open mapping theorem implies that \(U\) is into and onto, hence \begin{align*}U^{-1}: &X\to l^p\\
        &f_jQ_j\mapsto \{c_j\beta_j\}
    \end{align*}
        is bounded and \[M_\beta=U^{-1}\circ (S|_X),\]
    then \[\pi_r(M_\beta)\leq \|U^{-1}\|\pi_r(S|_X)\les\pi_r(S).\]
\end{proof}

\begin{lem}\label{multiplier2}
    let \(\mu\) be positive Borel measure on \(\bbD\) and \(\om\in \cR\). If \(J_\mu:\awp\to L^p(\mu)\) is \(r\)-suming, then \(M_\beta\) is \(r\)-summing.
\end{lem}
\begin{proof}
    Let \[r_n(t)=sign(\sin(2^n\pi t)),\quad t\in [0,1]\] then \(r_n(t)\) be an orthogonal system on \(L^1((0,1),dt)\) with 
    \[<r_m,r_n>=\int_0^1r_n(t)r_m(t)dt=\delta_{n}^m.\] If we set \(r=\{r_j\}_j\), then it is a Rademacher function on probability space 
    \(L^1(0,1)\) with probability measure \(dt\).
    We consider the following maps.
    \begin{align*}
        l^p\xrightarrow{\vphi_t}\awp\xrightarrow{J_\mu}L^p(\mu)\xrightarrow{U_t}L^p(\mu).
    \end{align*}
for each \(t\in [0,1]\)By Lemma \ref{decomposition}, \[\phi_t:a\in l^p\longmapsto \sum_ja_jr_j(t)f_j\] is well defined. By condition \(J_\mu\) is \(r-\)summing, hence the second map is well defined. 
For each \(t\in[0,1]\), we define 
\[U_t: f\in L^p(\mu)\longmapsto \sum_j r_n(j)\chi_{Q_j}f\in L^p(\mu),\]
then 
\begin{align*}
   \int_\bbD \bigg|\sum_j r_n(j)\chi_{Q_j}f\bigg|^pd\mu&=\sum_j\int_\bbD|r_n(j)f|^p\chi_{Q_j}d\mu\\
   &=\sum_j\int_{Q_j} |f|^pd\mu.
\end{align*}
Hence, \(U_t\) is an isometry with \[\|U_t\|=1.\]
Let \(\{e_k\}\) be a normalized base of \(l^p\), by \eqref{eqT}, it is not difficult to see 
\begin{align*}
    Te_k=f_k\chi_{Q_k}&=f_j\chi_{Q_j}\delta_k^j\\
    &=\int_0^1\left(\sum_jr_j(t)\chi_{Q_j}\right)f_kr_k(t)dt\\
   &= \int_0^1U_t\circ J_\mu\circ\vphi_tdt(e_k)
\end{align*}
hence \(T=\int_0^1U_t\circ J_\mu\circ\vphi_tdt\), notice that \(dt\) is probability measure and \(\pi_r\) is a convex function, by Lemma \ref{decomposition}, 
\begin{align*}
    \pi_r(M_\beta)&\asymp\pi_r(T)\\
    &=\pi_r\left(\int_0^1U_t\circ J_\mu\circ\vphi_tdt\right)\\
    &\leq\|U_t\|\int_0^1\pi_r(J_\mu)dt\|\vphi_t\|\\
    &\les\pi_r(T).
\end{align*}
\end{proof}

\begin{remark}\label{estimates}
    Lemma \ref{multiplier2} and Lemma \ref{multiplier}  provide some necessary conditions on \(r\)-summing operators.
    If \(J_\mu:\awp\to L^p(\mu)\) is \(r\)-summing, then 
    \[\pi_r(J_\mu)\asymp\pi_r(M_\beta).\]
    By lemma \ref{multiplier},
    \begin{enumerate}
        \item If \(p\leq r,\) then 
        \[\left(\int_\bbD\frac{d\mu(z)}{(\whom(z)(1-|z|))}\right)^\frac{1}{p}\asymp\left(\frac{\mu(D_k)}{\whom(a_k)(1-|a_k|)}\right)^\frac{1}{p}\les\pi_r(J_\mu).\]
        \item If \(p'\leq r\leq p\), then
        \[\left(\sum_k\left(\frac{\mu(D_k)}{\whom(a_k)(1-|a_k|)}\right)^\frac{r}{p}\right)^\frac{1}{r}\les \pi_r(J_\mu).\]
        \item If \(1\leq r\leq p'\), then 
         \[\left(\sum_k\left(\frac{\mu(D_k)}{\whom(a_k)(1-|a_k|)}\right)^\frac{1}{p-1}\right)^\frac{p-1}{p}\les\pi_r(J_\mu).\]
    \end{enumerate}
\end{remark}

\section{The case \(1<p\leq2\)}

We reminder that the cotype of \(A_\om^p\) and \(L^q(\mu)\) is \(2\), hence \(\Pi_2(A_\om^p, L^q(\mu))=\Pi_1(A_\om^p, L^q(\mu))\), so we pay our attention to the \(2-\)summing operators.
\begin{lemm}\label{hilbert}
    Let \( X \) be a Banach space and \( H \) be a Hilbert space. We assume that the operator \( T : X \to H \) is such that its adjoint \( T^* : H \to X^* \) is \( r \)-summing for some \( r \geq 1 \). Then \( T \) is 1-summing with

    \[
    \pi_1(T) \lesssim \pi_r(T^*).
    \]
\end{lemm}
\begin{theorem}\label{2summing}
   Let \(1<p<2\), \(\mu\) be positive Borel measure on \(\bbD\), \(\om\in\cD\), then the following results are equivalent. 
   \begin{enumerate}
       \item \(J_\mu: A_\om^p\to L^q(d\mu)\) is \(2-\)summing.
       \item \[\int_\bbD\frac{\mu(\Delta(z,r))^\frac{p'}{2}}{\left(\whom(\xi)(1-|\xi|)\right)^{\frac{2}{p}}}\om(\xi)dA(\xi)<\infty.\]
       \item There exists an operator \(U: L_\om^p\to L^2(\mu)\) is absolutely summing.
   \end{enumerate}
\end{theorem}

\begin{proof}
     We prove the theorem in its nature order.  Supposing that \(J_\mu\) is \(2-\)summing. Set \[f_{\xi}(z)=\left(\frac{1-|\xi|^2}{1-\overline{\xi}z}\right)^\gamma\frac{1}{(\whom(\xi)(1-|\xi|))^\frac{1}{2}}.\] 
    Since \(\om\in\cD\), By Lemma \ref{hatpropety},
    \begin{align*}
       \|f_\xi\|_{A_\om^2}^2 \asymp \frac{(1-|\xi|)^{^{2\gamma-1}}}{\whom(\xi)}\int_\bbD\frac{\om(z)}{|1-\overline{\xi}z|}dA(z)\asym1, \,\,\xi\in\bbD.
    \end{align*}
    This implies that \(f_\xi\in A_\om^p\).
Next we show that \(A_\om^{p*}=A_\om^{p'}\) under the pair \(<f,g>=\int_\bbD f\overline{g}\om dA\). One part of the proof is easy. It is obtained immediately from H\"older's inequality that the functional \(\phi_g\) defined above is bounded, hence \(\|\phi_g\|\) is bounded. 
Conversely, the Hahn-Banach theorem yields that any function \(\phi\in A_\om^{p*}\) can be extended to a functional \(\Phi\in L_\om^{p*}\) without norm increasing, thus \(\|\phi\|=\|\Phi\|\). By Riesz-Fisher theorem, 
\[\Phi(f)=\int_\bbD f\overline{h}\om dA\] for some \(h\in L_\om^{p'}\) with \(\|\Phi\|=\|h\|_q\). Let \(g=P_\om h\), where \(P_\om\) is the Bergman projection.
   Since \(\om\in\cD\), by \cite[THeorem 5]{Pelaez2021}, \(P_\om: L_\om^p\to A_\om^p\) is bounded, hence there exists a positive constant \(C\) such that \(\|P_\om h\|\leq C\|h\|\). 
   Now for \(f\in A_\om^p\), by Fubini's theorem, the boundedness of maximal Bergman projection and the self-adjointness of \(P_\om\) yiels that
   \begin{align*}
       \phi(f)=\Phi(f)&=\int_\bbD f\overline{h}\om dA\\
       &=\int_\bbD P_\om f\overline{h}\om dA\\
       &=\int_\bbD f\overline{P_\om h}\om dA=\int_\bbD f\overline{g}\om dA=\phi_g(f),
   \end{align*}
   and \(\|g\|_{A_\om^q}\leq C\|\phi\|.\)
   Hence, for each \(F\in B(A_\om^{p*})\), there exists \(h\in B(L_\om^{p'})\) such that 
\[F(f_\xi)=\int_\bbD f_\xi(z)\overline{h(z)}\om(z)dA(z)=\overline{g(\xi)},\] where \(P_\om h=g.\) 
Then 
\begin{align*}
    \int_\bbD |F(f_{\xi_k})|^{p'}\om dA&=\int_\bbD |\overline{g(\xi)}|^{p'}\om dA\\
    &\leq\|g\|_{A_\om^{p'}}^{p'}\leq\|P_\om\|^{p'}\|h\|_{L_\om^{p'}}^{p'}.
\end{align*}
Therefore,\[\sup_{F\in B(A_\om^{p*})}\int_\bbD|F(f_\xi)|^{p'}\om dA\leq \|P_\om\|^{p'}.\] 
Notice that \(1<p<2\implies p'>2\), by the definition of \(p'-\)summing operators, 
\[\|J_uf_\xi\|_{L^2(\mu)}\leq \pi_{p'}(J_\mu)\sup_{F\in B(A_\om^{p*})} |F(f_\xi)|,\] it implies that 
\begin{align*}
\pi_2(J_\mu)\geq \pi_{p'}(J_\mu)&\geq\frac{1}{\|P_\om\|}\left(\int_\bbD \|f_\xi\|_{L^2(\mu)}^{p'}\om(\xi)dA(\xi)\right)^{\frac{1}{p'}}\\
&\gtr\frac{1}{\|P_\om\|}\left(\int_\bbD\left(\int_\bbD\frac{d\mu(z)}{\whom(\xi)(1-|\xi|)}\right)^{\frac{p'}{2}}\om(\xi)dA(\xi)\right)^{\frac{1}{p'}}.
\end{align*}
\((2)\implies(3)\). 
For \(f\in A_\om^p\), \(p'>2\), Let 
\begin{align*}
    U:&L_\om^p\longrightarrow L^2(\mu),\\
    &f\longmapsto \int_\bbD\frac{f(\xi)}{(\whom(\xi)(1-|\xi|))^{\frac{1}{p}}}\om(\xi)dA(\xi).
\end{align*}  
Next we show that \(U\) is an operators satisfies the condition.
 Indeed, by Minkowski's inequality,
\begin{align*}
    \|Uf\|_{L^2(\mu)}&\leq\int_\bbD\bigg\|\frac{f(\xi)}{(\whom(\xi)(1-|\xi|))^{\frac{1}{p}}}\bigg\|_{L^2(\mu)}\om(\xi)dA(\xi)\\
    &\leq\left(\int_\bbD|f(\xi)|^{p}\om(\xi)dA(\xi)\right)^{\frac{1}{p}}\left(\int_\bbD\left(\int_\bbD\frac{d\mu(z)}{(\whom(\xi)(1-|\xi|))^{\frac{2}{p}}}\right)^{\frac{p'}{2}}\om(\xi)dA(\xi)\right)^{\frac{1}{p'}}\\&
    <\infty.
\end{align*}
Let 
\begin{align*}
    R: & L^2(\mu)\longmapsto  L_\om^{p'}\\
    &f\longmapsto \int_\bbD\frac{f(\xi)}{(\whom(\xi)(1-|\xi|))^{\frac{1}{p}}}d\mu(\xi)
\end{align*}
By an argument of duality, 
\[\sup_{\|f\|_{L^2(\mu)}\leq 1}\bigg|\int_\bbD\frac{f(\xi)}{(\whom(\xi)(1-|\xi|))^{\frac{1}{p}}}d\mu(\xi)\bigg|=\left(\int_\bbD\frac{d\mu(\xi)}{(\whom(\xi)(1-|\xi|))^\frac{2}{p}}\right)^\frac{1}{2},\]
hence 
\[\int_\bbD\sup_{f\in B_{L^2(\mu)}}|R(f)|^{p'}\om dA<\infty,\]
it implies that \(J_\mu\) is order bounded, by Lemma \ref{order1}, \(J_\mu\) is \(p'-\)summing.  An direct computation yields that 
\(R^{**}=U^*=R\), then by Lemma \ref{hilbert},  \(U=R^*\) is absolutely summing and 
\[\pi_1(U)\les\pi_{p'}(R=U^*)\les\left(\int_\bbD\left(\int_\bbD\frac{d\mu(\xi)}{(\whom(\xi)(1-|\xi|))^{\frac{2}{p}}}\right)^{\frac{p'}{2}}\om(z)dA(z)\right)^{\frac{1}{p'}}.\]

\((3)\implies (1)\). For \(f\in\awp\), we consider the following maps, 
\[
\awp \xrightarrow{i_d} L_\om^p \xrightarrow{U} L^2(\mu), 
\]
\(J_\mu=U\circ id\). Since \(\awp\) is the restriction  \(L_\om^p\) to analytic functions, hence \(i_d\) is bounded. By an ideal and monotone property
\[\pi_2(J_\mu)\leq\pi_1(J_\mu)\leq\pi_1(U)\|i_d\|\leq\pi_1(U)<\infty,\]
  the proof is completed.
\end{proof}    
\begin{remark}
    We remark that the above result is also holds for \(p=2\). More precisely, \(J_mu\) is \(r-\)summing operator on Hilbert space \(A_\om^2\) for some \(r\geq 1\) if and only if \(J_\mu\) is Hilbert-Schmidt operator.  Since \(e_j=\{\frac{z^j}{\sqrt{2\om_j}}\}\) forms an orthonormal basis of \(A_\om^2\), hence by \cite[eq. 20]{Pelaezprojection2016}, Lemma \ref{hatpropety} and \ref{chekproperty},
    \begin{align*}
        \|J_\mu\|_{H-S}&=\left(\sum_j\|J_\mu(e_j)\|_{L^2(\mu)}^2\right)^{\frac{1}{2}}\\
        &=\left(\sum_{j}\int_\bbD\frac{|z|^{2j}}{2\om_j}d\mu(z)\right)^{\frac{1}{2}}\\
        &\asymp\left(\int_\bbD\frac{d\mu(z)}{\whom(z)(1-|z|)}\right)^\frac{1}{2}<\infty.
        \end{align*}
      \end{remark}

The following results is a necessary condition for \(r-\)summing operators for all \(1<p,q\leq 2\).
\begin{theorem}\label{necessary}
    let \(1<p,q\leq 2\), \(\om\in\cD\), \(\mu\) be positive Borel measure on \(\bbD\). If \(J_\mu: \awp\to L^q(\mu)\) is \(r-\)summing, then \(\frac{d\mu(z)}{(\whom(z)(1-|z|))^{\frac{q}{2}}}\) is finite on \(\bbD\) almost everywhere.
\end{theorem}

\begin{proof}
    Since \(1<p,q\leq 2\), both \(\awp\) and \(L^q(\mu)\) have cotype \(2\), then by section \(1\),
    \[\Pi_1(\awp,L^q)=\Pi_r(\awp,L^q),\quad r\geq 1.\]
    Let \[F_t(z)=\sum_{n=0}^N r_n(t)\frac{z^n}{\sqrt{2\om_n}},\] where \(r_n(t)\) is Rademacher random variables on \(L^1((0,1),dt)\).
    It follows from Fubini's theorem and Khinchine’s inequality that
    \begin{align}\label{eq1}
        \begin{aligned}
            &\int_0^1\left|\int_{ \bbD} \sum_{n=0}^N r_n(t) \frac{z^n}{\sqrt{2 \omega_n}}\right|^q d \mu(z) d t \\
            & =\int_{ \bbD} \int_0^1\left|\sum_{n=0}^N r_n(t) \frac{z^n}{\sqrt{2 \omega_n}}\right|^q d t d \mu(z)\\
            &\asymp \int_{\bbD}\left(\sum_{n=0}^N \frac{|z|^{2 n}}{2 \omega_n}\right)^{\frac{q}{2}} d \mu(z).  
        \end{aligned}
         \end{align}
    For any \(F\in B_{\awp^*}\), there exist a function \(g\in A_\om^{p'}\) such that  
    \[F(F_t)=<F_t, g>=\int_\bbD{F_t}(z)g(z)\om(z)dA(z).\] Let \(g(z)=\sum\limits_{n=0}^N\widehat{g(n)}z^n\), then
    \(F(F_\xi)=\sum_{n=0}^N\sqrt{2\om_n}\widehat{g(n)}r_n(t).\) By Fubini's theorem and Khintchine’s inequality again, 
    \begin{align}\label{eq2}
        \begin{aligned}
            \int_0^1\left|F\circ F_{t}\right|^q d t & =\int_0^1\left|\sum_{n=0}^N \sqrt{2 \omega_n} g(n)\right|^q d t \\
            & \approx\left(\sum_{n=0}^N 2 \omega_n|\widehat{g(n)}|^2\right)^{\frac{q}{2}} \\
            & \lesssim\left(\sum_{n=0}^{\infty} 2 \omega_n|\widehat{g(n)}|^2\right)^{\frac{9}{2}} \\
            & =\|g\|_{A_\omega^2}^q \lesssim 1.
        \end{aligned}
       \end{align}
    Apply Lemma \ref{propsumming} to \(\awp-\) valued random variables \(t\mapsto F_t\), 
    \begin{equation}\label{eq3}
        \int_0^1\|J_\mu\circ F_t\|_{L^q(\mu)}^qdt\leq \pi_q(J_\mu)^q\sup_{F\in B_{\awp^8}}\int_0^1|F\circ F_t|^qdt
    \end{equation}
    Combining with \eqref{eq1} \eqref{eq2}  \eqref{eq3}, we have 
\[\pi_q(J_\mu)^q\gtr\int_\bbD\left(\sum_{n=0}^N\frac{|z|^{2n}}{2\om_n}\right)^{\frac{q}{2}}d\mu(z).\] 
By Lemma \ref{hatpropety} , let \(N=1, p=1 \) in \cite[eq. 20]{Pelaezprojection2016}, we obtain that  
\[\int_\bbD\frac{d\mu(z)}{(\whom(z)(1-|z|))^{\frac{q}{2}}}<\infty,\]
the proof is completed.
\end{proof}

Now we can state the main result of this section.

\begin{theorem}\label{mainth2}
    Let \(1<p,q\leq 2\), \(\mu\) be a positive Borel measure on \(\bbD\), \(\om\in\cD\). Then the following statements are equivalent.
    \begin{enumerate}
        \item \(J_\mu:\awp\to L^q(\mu)\) is \(r-\) summing.
        \item     \(J_\mu:\awp\to L^q(\mu)\) is \(2-\) summing.
        \item \(A_\om^{\frac{p}{2-p}}\to L^{\frac{q}{2}}(\mu)\) is bounded.
    \end{enumerate}
\end{theorem}
Before giving the proof of Theorem \ref{mainth2}, we also need the following tools. 
\begin{lemm}\cite[Lemma 7.5]{Lefevre2018}\label{lemmainf} 
    Let $\sigma > 0$, $(\Omega, \Sigma, \mu)$ be a measure space and $h: \Omega \to [0, +\infty)$ be a measurable function. Then

\[
\inf \left\{ \int_{\Omega} \frac{h}{F} \, d\mu : F \in L^{\sigma}(\mu), F \geq 0, \int_{\Omega} F^{\sigma} \, d\mu \leq 1 \right\} = \left( \int_{\Omega} h^{\sigma/(\sigma+1)} \, d\mu \right)^{(\sigma+1)/\sigma}.
\]
\end{lemm}
\begin{lemm}\cite[Lemma 7.6]{Lefevre2018}\label{lemmasqt}
    Let \( 1 \leq q < 2 \) and let \( s > 1 \) be such that \( 1/s + 1/2 = 1/q \). Let \( X \) be a Banach space, and \( T: X \to L^q(\mu) \) a bounded operator. The necessary and sufficient condition for \( T \) to be a 2-summing operator is that there exists \( F \in L^s(\mu) \), with \( F \geq 0 \) \(\mu\)-a.e., such that \( T: X \to L^2(\nu) \) is well defined and 2-summing, where the measure \(\nu\) is the measure defined by  
\[
d\nu(z) = \frac{1}{F(z)^2} \, d\mu(z).
\]

Moreover, we have  
\begin{align}
    \begin{aligned}
        & \pi_2(T: X \to L^p(\mu)) \\
        &\asymp \inf \left\{ \pi_2(T: X \to L^2(\nu)) : d\nu = d\mu/F^2, \, F \geq 0, \, \int F^s \, d\mu \leq 1 \right\}.   
    \end{aligned}
  \end{align}
\end{lemm}

\begin{lemm}\cite{Diestel1995}\label{kyfan}
    Let \( E \) be a Hausdorff topological vector space, and \( F \) be a compact convex subset of \( E \). Let \( M \) be a set of functions on \( F \) with values in \((-\infty, +\infty]\) having the following properties:  
    \begin{enumerate}
        \item Each \( f \in M \) is convex and lower semi-continuous.  
        \item If \( g \in \text{conv}(M) \), the convex closure of \( M \), there is a \( f \in M \) with \( g(x) \leq f(x) \), for all \( x \in F \).  
        \item There is an \( r \in \mathbb{R} \) such that for each \( f \in M \) has a value less and equal to \( r \). 
    \end{enumerate}
    Then, there exists at least one \( x_0 \in F \), such that \( f(x_0) \leq r \) for all \( f \in M \).
 \end{lemm}
For \(f\in\cH(\bbD)\), \(|B(f)|^p\)  may not be subharmonic, but the following lemma shows that although \(|Bf|^p\) is not subharmonic, it still possesses properties similar to those of subharmonic functions. 
\begin{lemm}
Suppose \(\om\in \cR\), \(0<p<\infty\), \(r>0\), \(f\geq 0\). Then there holds 
\[(Bf)^q(z)\lesssim\frac{1}{\om(\Delta(z,r))}\int_{\Delta(z,r)}(Bf)^q(w)\om(w)dA(w).\]
\end{lemm}
\begin{proof}
    We suppose that \(p=1\). By Lemma \ref{hatpropety} abd Fubini's theorem,
\begin{align*}
    &\frac{1}{\om(\Delta(z,r))}\int_{(\Delta(z,r))}(Bf)(w)\om(w)dA(w)\\
    &=\frac{1}{\om(\Delta(z,r))}\int_{(\Delta(z,r))}\int_\bbD f(u)|b_w^\om(u)|^2\om(u)dA(u)\om(w)dA(w)\\
    &\asymp \frac{1}{\om(\Delta(z,r))}\int_{(\Delta(z,r))}\int_\bbD f(u)\om(\Delta(w,r))|B_w^\om(u)|^2\om(u)dA(u)\om(w)dA(w)\\
    &\asymp \int_\bbD\frac{f(u)}{\widehat{\om}(z)(1-|z|)}\int_{D(z,r)}\om(\Delta(z,r))|B_z^\om(u)|^2\om(w)dA(w)\om(u)dA(u)\\
    &\asymp\int_\bbD f(u)\om(\Delta(z,r))|B_z^\om(u)|^2\om(u)dA(u)\\
    &\asymp Bf(z).
\end{align*}
The results holds for \(p=1\). Next we show that the result holds for \(p>1\). Since 
\[(Bf)(z)\asymp \frac{1}{\om(\Delta(z,r))}\int_{\Delta(z,r)}(Bf)(w)\om(w)dA(w).\]
For \(f(x)=x^p\) is a convex function and for \(w\in \Delta(z,r)\), \(\frac{\om(w)dA}{\om(\Delta(z,r))}\) is a probability measure on \(\Delta(z,r)\),  thanks to the Jensen's inequality yields 
\begin{align*}
    (Bf)^q(z)&\asymp \left(\int_{\Delta(z,r)}(Bf)(w)\frac{\om(w)dA(w)}{\om(\Delta(z,r))}\right)^q\\
   & \lesssim \frac{1}{\om(\Delta(z,r))}\int_{\Delta(z,r)}(Bf)^q(w)\om(w)dA(w).
\end{align*}
Thus the results holds for \(p>1\).
we end this proof with the case of \(0<p<1\). Fix \(z_0\in\bbD\), then 
\begin{align*}
&\frac{1}{\om(\Delta(z_0,r))}\int_{\Delta(z_0,r)}\int_\bbD f(u)|b_w^\om(u)|^2\om(u)dA(u)\om(w)dA(w)\\
&\asymp \frac{1}{\widehat{\om}(z_0)(1-|z_0|^2)}\int_\bbD f(u)\int_{\Delta(z_0,r)}|b_z^\om(u)|^2f(u)\om(u)dA(u)\om(w)dA(w)\\
&\asymp (Bf)(z).
\end{align*}
Hence \begin{align*}
    \sup_{z\in\Delta(z_0,r)} Bf(z)&\lesssim \frac{1}{\widehat{\om}(z_0)(1-|z_0|)}\int_{\Delta(z_0,r)}Bf(w)\om(w)dA(w)\\
&\frac{1}{\widehat{\om}(z_0)(1-|z_0|)}\left(\int_{\Delta(z_0,r)}Bf(w)\om(w)dA(w)\right)\sup_{w\in\Delta(z_0,r)}(Bf)^{1-q}(w).
\end{align*}
Hence \begin{align*}
    &(Bf)^q(z_0)\\
   & \leq\sup_{z\in\Delta(z_0,r)}(Bf)^q(z)\\
   &\lesssim\frac{1}{\widehat{\om}(z_0)(1-|z_0|^2)}\int_{\Delta(z_0,r)}(Bf)^q(w)\om(w)dA(w).
\end{align*}
The proof is completed.
\end{proof}

Now we can state the main result of this subsection.
\begin{theorem}
    Let \(1<p<\infty\), \(q>0\), \(\om\in\cD\). Then the following are equivalent.
    \begin{enumerate}
        \item \(B: L_\om^p\to L^q(\mu)\) is bounded. 
        \item \(J_\mu:\awp\to L^q(\mu)\) is bounded. 
    \end{enumerate}
\end{theorem}             
\begin{proof}
  Assuming that \(B\) is bounded, we consider the following maps.
  \begin{align*}
    \awp \xrightarrow{i_d} L_\om^p \xrightarrow{B} L^q(\mu),
  \end{align*} 
  where \(i_d\) is the natural embedding, for \(g\in A_\om^p\),
  \[Bg=g,\] hence, from the fact that \(J_\mu=B\circ i_d\) we have \((1)\implies (2)\).
Next we show that \((2)\implies (1)\). 
The case \(0<q<p<\infty\). Without loss of generality, we assume that \(f\geq 0\).
\[\left(B(f)(w)\right)^q\les\frac{1}{\whom(w)(1-|w|)}\int_{D(w,r)} (Bf)^q(z)\om(z)dA(z),\] then by Fubini's theorem and a direct computation yields
\begin{align}\label{eq5}
    \begin{aligned}
        \int_\bbD\left(B(f)(w)\right)^qd\mu(w)&\les\int_\bbD \frac{1}{\whom(w)(1-|w|)}\int_{D(w,r)} (Bf)^q(z)\om(z)dA(z)d\mu(w)\\
   & =\int_\bbD \frac{1}{\whom(w)(1-|w|)}\int_\bbD\chi_{D(w,r)}(z) (Bf)^q(z)\om(z)dA(z)d\mu(w)\\
   &\asymp\int_\bbD\left(B(f)(z)\right)^q\om(z)dA(z)\int_\bbD\frac{\chi_{D(z,r)(w)}}{\whom(z)(1-|z|)}d\mu(w)\\
   &\int_\bbD\left(B(f)(z)\right)^q\frac{\mu(D(z,r))}{\whom(z)(1-|z|)}\om(z)dA(z).
    \end{aligned}
   \end{align}
Let \[K(z)=\frac{\mu(D(z,r))}{\whom(z)(1-|z|)},\] combine \eqref{eq5} and H\"older's inequality we get 
\begin{align*}
    \int_\bbD \left(Bf(w)\right)^qd\mu(w)&\les\left(\int_\bbD\left(Bf(z)\right)^p\om(z)dA(z)\right)^{\frac{q}{p}}\|K\|_{L_{\om}^{\frac{p}{p-q}}}\\
    &\les\|f\|_{L_\om^p}^q\|K\|_{L_{\om}^{\frac{p}{p-q}}}\\
    &\asymp\|f\|_{L_\om^p}^q\|J_\mu\|_{\awp\to L^q(\mu)}
\end{align*}
The case \(p>1\), by the case \(q<p\) and \(q>0\), we need only consider the case \(p\leq q\).

Let \(\{a_k\}\) be an \(r\)-lattice of \(\bbD\), then 
\begin{align}\label{eq6}
\begin{aligned}
    \int_\bbD\left(B(f)(w)\right)^qd\mu(w)&\les\sum_{k}\int_\dak\left(\frac{1}{\whom(w)(1-|w|)}\int_{D(w,r)}\left(B(f)(z)\right)^q\om(z)dA(z)d\mu(w)\right)\\
    &\les\sum_{k}\int_\dak\left(\frac{1}{\whom(a_k)(1-|a_k|)}\int_{D(a_k,2r)}\left(B(f)(z)\right)^q\om(z)dA(z)\right)d\mu(w).   
\end{aligned}
\end{align}
By Lemma \eqref{hatpropety} and H\"older's inequality
\begin{align*}
    Bf(z)&=<fb_z^\om,b_z^\om>\\
    &=\int_\bbD f(\xi)|b_z^\om(\xi)|^2\om(\xi)dA(\xi)\\
    &\leq\|f\|_{\awp}\left(\int_\bbD|\bzw|^{2p'}\om(\xi)dA(\xi)\right)^{\frac{1}{p'}}\\
    &=\|f\|_{\awp}\frac{1}{\|B_z^\om\|_{A_\om^2}^2}\left(\int_\bbD|B_z^\om(\xi)|^{2p'}\right)^\frac{2}{2p'}\\
    &\asymp\frac{\om(S(z))}{\om(S(z))^{2-\frac{1}{p'}}}\|f\|_{\awp}\\
    &\asymp\frac{\|f\|_{\awp}}{(\whom(z)(1-|z|))^{\frac{1}{p}}}.
\end{align*}
This fact combine with \eqref{eq6} and \(\whom(a_k)(1-|a_k|)\asymp\whom(z)(1-|z|) \) for \(z\in D(a_k,2r)\), Applying Theorem \ref{cm}, we get 
\begin{align*}
    \int_\bbD\left(B(f)(w)\right)^qd\mu(w)&\les\sum_k\frac{\mu(\dak)}{\whom(a_k)(1-|a_k|)}\int_{D(a_k,2r)}\left(B(f)(z)\right)^q\om(z)dA(z)\\
    &\les\sum_k\frac{\mu(\dak)}{\whom(a_k)(1-|a_k|)}\int_{D(a_k,2r)}\left(B(f)(z)\right)^{q-p}\left(B(f)(z)\right)^p\om(z)dA(z)\\
    &\les\sup\frac{\mu(\dak)}{(\whom(a_k)(1-|a_k|))^\frac{q}{p}}\int_\bbD\left(B(f)(z)\right)^p\om(z)dA(z)\|f\|_{\awp}^{q-p}\\
    &\les\|f\|_{L_\om^p}^{q-p}\|Bf\|_{L_\om^p}^p\les\|f\|_{\awp}^q
\end{align*}
\end{proof}

\begin{proof}[Proof of Theorem \ref{mainth2}]
    We only need show that \((1) \iff (3)\). Let \(s=\frac{q}{2-q}\), then \(\frac{1}{2s}+\frac{1}{2}=\frac{1}{q}\).
    Applying Lemma \ref{lemmainf}, we have 
    \begin{align*}
        \pi_2^2\left(J_\mu:\awp\to L^q(\mu)\right)&\asymp \inf_{F\in B_{L^{2s}(\mu)}^+}\left\{\pi_2^2\left(J_\mu:\awp\to L^2(\frac{\mu}{F^2})\right)\right\}\\
        &\asymp\inf_{f\in B_{L^s(\mu)}^+}\left(\int_\bbD\left(\int_\bbD\frac{d\mu(z)}{(\whom(z)(1-|z|))^{\frac{2}{p}}f(z)}\right)^{\frac{p'}{2}}\om(\xi)dA(\xi)\right)^{\frac{2}{p'}}
    \end{align*}
    Replacing \(F^2\) by \(f\) and \(s\) by \(\frac{s}{2}\), let \(t=\frac{p}{2-p}\), then \(\frac{1}{t}+\frac{1}{\frac{p'}{2}}\). 
By an argument of duality, we get 
\begin{align}\label{pi2}
    \pi_2^2\left(J_\mu:\awp\to L^q(\mu)\right)\asymp\inf_{f\in B_{L^s(\mu)}^+}\sup_{g\in B_{L_\om^t}^+}\int_\bbD\int_\bbD\frac{g(\xi)d\mu(z)}{(\whom(z)(1-|z|))^{\frac{2}{p}}}\om(\xi)dA(\xi).
\end{align}
Let \[A=\inf_{f\in B_{L^s(\mu)}^+}\sup_{g\in B_{L_\om^t}^+}\int_\bbD\int_\bbD\frac{g(\xi)d\mu(z)}{(\whom(z)(1-|z|))^{\frac{2}{p}}f(z)}\om(\xi)dA(\xi),\]
\[B=\sup_{g\in B_{L_\om^t}^+}\inf_{f\in B_{L^s(\mu)}^+}\int_\bbD\int_\bbD\frac{g(\xi)d\mu(z)}{(\whom(z)(1-|z|))^{\frac{2}{p}}f(z)}\om(\xi)dA(\xi).\]
We claim that \(A=B\).  Indeed, by the definition of supremum and infimum, it is obviously to see that   \(A\geq B\).

Next we show that \(A\leq B\).
We assume that \(B\) is finite, let a family of functionals 
\[M=\left\{\phi_g|g\in B_{L_\om^t(\mu)}^+\right\}\]
where, for \(f\in B_{L^s(\mu)}^+\)
\[\phi_g(f)=\int_\bbD\int_\bbD\frac{g(\xi)d\mu(z)}{(\whom(z)(1-|z|))^{\frac{2}{p}}f(z)}\om(\xi)dA(\xi).\] 
Thanks to the Alaoglu's theorem, the set \(F=B_{L_\om^s(\mu)}\) is convex compact under the weak topology,  and \(M\) can be regarded as a functional on \(F\). By the linearity of \(g\),  the set \(M\) is convex, hence \(conv(M)=M\). 

For \(\phi_g\in M\), \(\phi_g\) is convex because of \(x\mapsto \frac{1}{x}\) is convex and lower semi-continuous. 
Let \[K=\{f\in F| \phi_g\leq\lambda\},\] we claim that \(K\) is closed under the weak topology. Since \(F\) be a Banach space, then let \(\{f_n\}\in F\) converging to \(f\in L^s(\mu)\), up to an extraction, we can assume that \(f_n\) is also pointwise to converging a.e. to \(f\). Hence, by Fatou's lemma, 
\[\phi_g(f)\leq\liminf_{n\to\infty}(f_n)\leq\lambda.\] 
Thus it satisfies conditions \((1)\) and \((2)\) in  Lemma \ref{kyfan}.
Let \(r=B+\varepsilon\), then it satisfies condition \((3)\) Lemma \ref{kyfan}. In fact, by the definition of \(B\), for \(g\in B_{L_\om^t}^+\), 
\[\inf_{f\in B_{L^s(\mu)}^+}\phi_g(f)<B+\varepsilon=r,\] hence there exists a \(f_0\) such that \(\phi_g(f_0)\leq r.\) 

By Lemma \ref{kyfan}, there exists a \(f'\in B_{L^s(\mu)^+}\) such that \(\phi_g(f')\leq B+\varepsilon\), this implies that 
\[A\leq B+\varepsilon,\] let \(\varepsilon\to 0^+\), the claim is proved.

Now we back to theorem \ref{mainth2}. By the above claim, Lemma \ref{lemmainf} and \eqref{pi2}, 
\begin{align*}
    \pi_2^2\left(J_\mu:\awp\to L^q(\mu)\right)&\asymp\inf_{f\in B_{L^s(\mu)}^+}\sup_{g\in B_{L_\om^t}^+}\int_\bbD\int_\bbD\frac{g(\xi)d\mu(z)}{(\whom(z)(1-|z|))^{\frac{2}{p}}}\om(\xi)dA(\xi)\\
    &\asymp\sup_{g\in B_{L_\om^t}^+}\inf_{f\in B_{L^s(\mu)}^+}\int_\bbD\int_\bbD\frac{g(\xi)d\mu(z)}{(\whom(z)(1-|z|))^{\frac{2}{p}}}\om(\xi)dA(\xi)\\
    &\asymp\sup_{g\in B_{L_\om^t}^+}\left\|\frac{B(g)(z)}{|B_z^\om(\xi)\om(S(z))\om(S(\xi))}\right\|_{L^{\frac{q}{2}}(\mu)}\\
    &\asymp\sup_{g\in B_{L_\om^t}}\|B(g)\|_{L^q(\frac{\nu}{2})}\\
     &\asymp \|J_\nu\|_{A_\om^t\to L^{\frac{q}{2}}(\nu)},
\end{align*}
the proof is completed.
\end{proof}

\section{The case \(2< p\leq q <\infty\)}
\subsection{\textbf{The case \(2< p\leq r\)}}

In this section we pay our attention to the case \(2< p\leq r\), let us see a simple case \(p=r\) first.

\begin{theorem}\label{p2summing}
    Let \(p> 2\), \(\mu\) be a finite positive measure on \(\bbD\), \(\om\in\cD\). Then \(J_\mu:\awp\to L^q(\mu)\) is \(p\)-summing if and only if \(\frac{d\mu}{(\whom(z)(1-|z|))^\frac{q}{p}}\) is a finite measure.
\end{theorem}
\begin{proof}
    By Lemma \ref{order2}, \(\frac{d\mu}{(\whom(z)(1-|z|))^\frac{q}{p}}\) is a finite measure implies that \(J-\mu\) is order bounded, hence \(J_\mu\) is \(p\)-summing. 

    Conversely, let \(\{c_j\}\) be a sequence in \(l^p\), \(\{a_j\}\) be lattice,  then by the proof of \cite[Proposition 14]{Pelaez2018}
    \[e_j=c_jb_{p,a_j}^\om\in\awp.\]
and for \(g\in A_\om^{p'}\)
    \begin{align}\label{sum1}
        \begin{aligned}
            \sum_j|<e_j,g>|^p\les 1.
        \end{aligned}
    \end{align}
    The definition of \(p\)-summing yields 
    \begin{equation}\label{psum}
        \sum_j\|e_j\|_{L^q(\mu)}^p\leq\pi_p^p(J_\mu)\sup_{g\in A_\om^{p'}}\sum_j|<e_j,g>|^p.
    \end{equation} 
 Since 
 \begin{align*}
    \begin{aligned}
        \sum_j\|e_j\|_{L^q(\mu)}^p&=\sum_j\left(\int_\bbD|e_j|^qd\mu(z)\right)^\frac{p}{q}\\
        &\geq\left(\int_\bbD \sum_j |c_jb_{p,a_j}^\om|^qd\mu(z)\right)^\frac{p}{q}.
    \end{aligned}
 \end{align*}
 By replacing \(c_j\) by \(r_j(t)c_j\), where \(r_j\) denotes the \(k-\) Rademacher random variables, then by Khinchine's inequality and Fubini's theorem yields
 \begin{align}\label{Khinchine}
    \begin{aligned}
        \sum_j\|e_j\|_{L^q(\mu)}^p&\geq\left(\int_\bbD\left(\sum_j|c_j|^2|b_{p,z_j}^\om|^2\right)^\frac{q}{2}d\mu(z)\right)^\frac{p}{q}\\
        &\gtr\left(\sum_j|c_j|^q\int_{\Delta(a_j,r)}|b_{p,a_j}^\om|^qd\mu(z)\right)^\frac{p}{q}.
    \end{aligned} 
 \end{align}
 By Lemma  and the fact \(\whom(a_k)(1-|a_k|)\asymp\whom(z)(1-|z|) \) for \(z\in\Delta(a_k,r)\)
 \begin{align*}
\int_{\Delta(a_j,r)}|b_{p,a_j}^\om|^qd\mu(z)&=\int_{\Delta(a_j,r)}\frac{|B_{a_j}^\om|^q}{\|B_{a_j}^\om\|_{\awp}^q}d\mu(z)\\
&\gtr\frac{|B_{a_j}^\om(a_j)|}{\om(S(a_j))^{q-\frac{p}{q}}}\int_{\Delta(a_j,r)}d\mu(z)\\
&\gtr\int_{\Delta(a_j,r)}\frac{1}{(\whom(a_j)(1-|a_j|))^\frac{q}{p}}d\mu(z)\\
&\gtr\int_{\Delta(a_j,r)}\frac{d\mu(z)}{(\whom(z)(1-|z|)^\frac{q}{p}}.
 \end{align*}
 Hence, by \eqref{Khinchine} and \eqref{psum}, we obtain that 
\[\int_\bbD\frac{d\mu(z)}{(\whom(z)(1-|z|))^\frac{q}{p}}<\infty.\]

\end{proof}

\begin{theorem}
    let \(2\leq p\leq r\), \(\mu\) be a finite positive measure on \(\bbD\). Then \(J_\mu:\awp\to L^p(\mu)\) is an \(r\)-summing if and only if  
\begin{equation}\label{2pr}
    \int_\bbD\frac{d\mu(z)}{(\whom(z)(1-|z|))^\frac{q}{p}}<\infty.
\end{equation}
\end{theorem}

\begin{proof}
    Since \(p\leq r\), hence \(J_\mu\) is \(p\)-summing must be \(r\)-summing. Then by Theorem \ref{p2summing} and \eqref{2pr}, 
    \[\pi_r(J_\mu)\leq\pi_p(J_\mu)\les\int_\bbD\frac{d\mu(z)}{(\whom(z)(1-|z|))^\frac{q}{p}}.\]
    The reverse inequality can be obtained by Remark \ref{estimates}.
\end{proof}

\subsection{\textbf{The case\(p'\leq r\leq p\)}}

In this section, we shall work with the case \(p'\leq r\leq p\) with \(p\geq 2\). Let \(\awp(E)=\cH(\bbD)\bigcap L_\om^p(E)\) with norm
\[\|f\|_{\awp(E)}=\left(\int_E|f(z)|^p\om(z)dA(z)\right)^\frac{1}{p},\] where \(E\) is a subset of \(\bbD\).

Given a family of Banach spaces \(\{X_k\}\), the space \(\bigoplus\limits_{l^p}X_k\) is equipped with the norm 
\[\|\{x_k\}\|_{\bigoplus_{l^p}X_k}=\left(\sum_k\|x_k\|_{X_k}^p\right)^\frac{1}{p},\]
where \(x_k\in X_k\) for every \(k\geq 1\).
\begin{lemm}\cite[Corollary 4.8]{Lefevre2018}\label{glue}
    Let \( p \geq 2 \), \( p' \leq r \leq p \). Assume that for every integer \( n \geq 1 \), we have bounded operators \( T_n : X_n \to Y_n \) for Banach spaces \( X_n \) and \( Y_n \). Consider the operator

\[
T : \bigoplus_{\ell^p} X_n \longrightarrow \bigoplus_{\ell^p} Y_n \\
(x_n)_n \longmapsto (T_n(x_n))_{n \geq 1}.
\]

Then \( T \) is \( r \)-summing if and only if each \( T_n \) is \( r \)-summing and

\[
\sum_{n \geq 1} \pi_r(T_n)^r < +\infty.
\]

Moreover, we have

\[
\pi_r(T) \approx \left( \sum_{n \geq 1} \pi_r(T_n)^r \right)^{1/r}.
\]
\end{lemm}
\begin{lemm}\label{partembedding}
    Let \(1\leq p\leq q<\infty\), \(\mu\) be a positive Borel measure on \(\bbD\), \(\om\in\cD\), \(0<t<\frac{1}{2}\), \(\{a_k\}\) be a \(t\)-lattice, \(D_k\) be a regular lattice. Then \[J_\mu:\awp(\Delta(a_k,2t))\hookrightarrow L^p(D_k,\mu)\] is bounded if and only if 
    \[\frac{\mu(D_k)}{(\whom(a_k)(1-|a_k|))^\frac{q}{p}}<\infty.\]
\end{lemm}

\begin{proof}
  \begin{align*}
    \int_{D_k}|f(z)|^qd\mu(z)&\leq\left(\sup_{z\in D_k}|f(z)|^p\right)^\frac{q}{p}\mu(D_k)\\
    &\les\sup_{z\in D_k}\frac{1}{(\whom(z)(1-|z|))^\frac{q}{p}}\left(\int_{\Delta(z,t)}|f(z)|^p\om(z)dA(z)\right)^\frac{q}{p}\mu(D_k)\\
    &\les\sup_{z\in D_k}\frac{1}{(\whom(z)(1-|z|))^\frac{q}{p}}\left(\int_{\Delta(a_k,2t)}|f(z)|^p\om(z)dA(z)\right)^\frac{q}{p}\mu(D_k)\\
    &\asymp\frac{\mu(D_k)}{(\whom(a_k)(1-|a_k|))^\frac{q}{p}}\|f\|_{\awp(\Delta(a_k,2t))}^q\\
    &\asymp\frac{\mu(D_k)}{\om(D_k)^\frac{q}{p}}\|f\|_{\awp(\Delta(a_k,2t))}^q.
  \end{align*}  
  This implies that 
  \[\|J_\mu\|\les\frac{\mu(D_k)^\frac{1}{q}}{(\whom(a_k)(1-|a_k|))^\frac{1}{p}}.\]
  Conversely, by Lemma \ref{hatpropety}, Let \[f(z)=\left(\frac{1-|a_k|^2}{1-\overline{a_k}z}\right)^\gamma\frac{1}{(\whom(a_k)(1-|a_k|))^\frac{1}{p}}.     \]
  Applying Lemma \ref{hatpropety} again yields 
  \[\|f\|_{awp}\asymp\frac{(1-|a_k|)^{p\gamma-1}}{\whom(a_k)}\int_\bbD\frac{\om(z)}{|1-\overline{a_k}z|^{p\gamma}}\asymp 1.\]
  Hence 
  \begin{align*}
    \infty&>\|J_\mu\|\|f\|_\awp\\
    &\gtr\left(\int_{D_k}|f(z)|^qd\mu(z)\right)^\frac{1}{q}\\
    &\gtr\left(\int_{D_k}\bigg|\frac{1-|a_k|^2}{1-\overline{a_k}z}\bigg|^{q\gamma}\frac{d\mu(z)}{(\whom(a_k)(1-|a_k|))^\frac{q}{p}}\right)^\frac{1}{q}\\
    &\gtr\frac{\mu(D_k)^\frac{1}{q}}{(\whom(a_k)(1-|a_k|))^\frac{1}{p}}\\
    &\asymp\frac{\mu(D_k)}{\om(D_k)^\frac{1}{p}}.
  \end{align*}
\end{proof}

\begin{theorem}\label{commutative}
    Let \(p\geq 2\), \(p'\leq r\leq p\), \(0<t<\frac{1}{4}\), \(\om\in\bbD\). Then \(J_\mu:\awp\to L^q(\mu)\) is \(r\)-summing if and only if 
\end{theorem}

\begin{proof}
    We consider the following maps. 
    \begin{center}
        \begin{tikzcd}[column sep=3cm,row sep=2cm]

            \awp \arrow[r, "J_\mu"] \arrow[r, "J_\mu"] \arrow[d, "i_1"'] & L^q(\mu)\\
            \bigoplus\limits_{l^p}\awp(\Delta(a_k,4t)) \arrow[r, "T"'] & \bigoplus\limits_{l^p}L^q(D_k,\mu) \arrow[u, "i_2"']
            \end{tikzcd}
    \end{center}
    Since \(J_\mu=i_2\circ T\circ i_1\), by the iedal property of \(\pi_r\), we have 
    \begin{equation}\label{rsum}
        \pi_r(J_\mu)\leq\|i_1\|\pi_r(T)\|i_2\|.
    \end{equation}
    We show that \(i_1\), \(i_2\) are bounded and \(T\) is \(r\)-summing.
For \(\bbD=\bigcup \Delta(a_k,4t)\) with \(0<t<\frac{1}{4}\), each point of \(\bbD\) belongs to at most \(N\) of the sets \(D(a_k,4t)\), thus the mapping \begin{equation}\label{i1}
    f\in\awp\longmapsto\left(f|_{\Delta(a_k,4t)}\right)_k\in\bigoplus \awp(D(a_k,4t))
\end{equation}  
is bounded. In fact,  for \(f\in\awp\), 
\begin{align*}
    \|f|_{\Delta(a_k,4t)}\|_{\bigoplus \awp(D(a_k,4t))}&=\left(\sum_k \|f\|_{\awp(\Delta(a_k,4t))}^p\right)^\frac{1}{p}\\
    &=\left(\sum_k\int_{\Delta(a_k,4t)}|f(z)|^p\om(z)dA(z)\right)^\frac{1}{p}\\
    &\les N^\frac{1}{p}\|f\|_{\awp}. 
\end{align*}
\eqref{i1} is just the map \(i_1\) in commutative diagram, thus \(i_1\) is bounded.

For \(f\in L^q(\mu)\), Since \(D_k\) is a partition of \(\bbD\), hence 
\begin{align*}
    \int_\bbD |f(z)|^qd\mu(z)&=\sum_k\int_{D_k}|f(z)|^qd\mu(z)\\
    &=\sum\|f|_{D_k}\|_{L^q(\mu)}^q\\
    &=\left(\left(\sum\|f|_{D_k}\|_{L^q(\mu)}^q\right)^\frac{p}{q}\right)^\frac{q}{p}\\
    &\leq\left(\sum\|f|_{D_k}\|_{L^q(\mu)}^p\right)^\frac{q}{p}.
\end{align*}
This implies that map 
\begin{equation}\label{i2}
    f\in L^q(\mu)\longmapsto f|_{D_k}\in\bigoplus_{l^p}L^q(D_k,\mu) 
\end{equation}
is bounded. In paritcular, if \(p=q\), then this map is an isometry. \eqref{i2} is just the map 
\(i_2\) in commutative diagram, hence \(i_2\) is bounded.

Next we show that \(T\) is $r$-summing. Let 
\[T_k:\awp(\Delta(a_k,4t))\longrightarrow L^q(D_k,\mu).\]
By the Lemma \ref{glue}, we need show that \(T_k\) is \(r-\)summing.
Consider the following maps,
\begin{center}
    \begin{tikzcd}[column sep=3cm,row sep=2cm]

        \awp(\Delta(a_k,4t)) \arrow[r, "T_k"] \arrow[d, "j_1"'] & L^q(D_k,\mu) \\
        H^\infty(\Delta(a_k,2t)) \arrow[r, "i"'] & \awp(\Delta(a_k,2t)) \arrow[u, "j_2"']
        \end{tikzcd} 
\end{center}
Since \(T_k=j_2\circ i\circ j_1\) with
 \begin{equation}\label{rrrsum}
    \pi_r(T_k)\leq \|j_1\|\pi_r(i)\|j_2\|,
 \end{equation} 
 so we need to show that \(j_1\) and \(j_2\) is bounded, \(i\) is \(r\)-summing. 

Observing that \(H^\infty\) is dense in \(A_\om^p\), hence \(i\) can be regarded as the identity map \(i_d: L^\infty\hookrightarrow L_\om^p\) restriction to analytic functions, hence \(i\) is absolutely summing and  \[\pi_1(i)\asymp 1.\]
Since \(i\) is \(1\)-summing, it must be \(r\)-summing, hence 
 \begin{equation}\label{1rsum}
    \pi_r(i)\leq\pi_1(i)\asymp 1.
 \end{equation}
Next we show that \(j_1\) is bounded.
\begin{align*}
    \sup_{z\in\Delta(a_k,2t)}|f(z)|&\les \sup_{z\in\Delta(a_k,2t)} \frac{1}{(\whom(z)(1-|z|))^\frac{1}{p}}\left(\int_{\Delta(z,2t)}|f(w)|^p\om(w)dA(w)\right)^\frac{1}{p}\\
    &\les\sup_{ \Delta(a_k,4t)}\frac{1}{(\whom(z)(1-|z|))^\frac{1}{p}}\left(\int_{\Delta(a_k,4t)}|f(w)|^p\om(w)dA(w)\right)^\frac{1}{p}\\
    &\asymp \frac{1}{(\whom(a_k)(1-|a_k|))^\frac{1}{p}}\|f\|_{\awp(\Delta(a_k,4t))}\\
    &\asymp\frac{1}{\om(D_k)}\|f\|_{\awp}.
\end{align*}
Thus \(j_1\) is bounded with norm 
\[\|j_1\|\asymp 1.\]
By Lemma \ref{partembedding}, \(j_2\) is bounded and by  \eqref{rrrsum}, \eqref{1rsum}, 
\begin{equation}\label{1tk}
    \pi_1(T_k)\les\frac{\mu(D_k)}{(\whom(a_k)(1-|a_k|))^\frac{q}{p}}.
\end{equation}
By \eqref{rsum} and Lemma \ref{glue},
\[\pi_r(j_\mu)\les\pi_r(T)\asymp\left(\sum_k\pi_r^r(T_k)\right)^\frac{1}{r}\leq\left(\sum_k\pi_1^r(T_k)\right)^\frac{1}{r}\]
we prove that the sufficiency.     
The proof of necessity can be obtained by Remark \ref{estimates}.
\end{proof}

\subsection{\textbf{The case\(1\leq r\leq p'\)}}
\begin{lemm}\cite[Corollary 9.2]{He2024}\label{cor}
    Let \( X_n \) and \( Y_n \) be Banach spaces. Let \( T_n : X_n \longrightarrow Y_n \) be a 1-summing operator.  
We assume that  

\[
\sum_{n \geq 1} \pi_1(T_n)^{p'} < \infty.
\]

Then the operator  

\[
T : x = (x_n)_n \in \bigoplus_{\ell p} X_n \longmapsto \left( T_n(x_n) \right)_{n \geq 1} \in \bigoplus_{\ell p} Y_n
\]

is 1-summing with  

\[
\pi_1(T) \lesssim \| \pi_1(T_n) \|_{\ell^{p'}}.
\]
\end{lemm}
\begin{theorem}
    Let \(p\geq 2\), \(1\leq r\leq p'\), \(\mu\) be a positive Borel measure on \(\bbD\). Then the following are equivalent.
   \begin{enumerate}
    \item \(J_\mu:\awp\to L^q(\mu)\) is \(1\)-summing.
    \item \(J_\mu:\awp\to L^q(\mu)\) is \(r\)-summing.
    \item \(\sum_k\left(\frac{\mu(D_k)}{\whom(a_k)(1-a_k)}\right)^\frac{1}{p-1}<\infty.\)
   \end{enumerate}
\end{theorem}
\begin{proof}
    We only provide an outline of a proof, we prove this theorem in natural order. 
    \((1)\implies (2)\) is trivial. 
    \((2)\implies (3)\). Since \(r\leq p'\), then by Theorem \ref{commutative}
    \[\infty>\pi_r(J_\mu)\geq\pi_{p'}(J_\mu)\gtr\left(\sum_k\left(\frac{\mu(D_k)}{\whom(a_k)(1-a_k)}\right)^\frac{1}{p-1}\right)^\frac{p-1}{p} .\]
    \((3)\implies (1)\). We consider the following commutative diagram. 
    Since \[\sum_k\left(\frac{\mu(D_k)}{\whom(a_k)(1-a_k)}\right)^\frac{1}{p-1}<\infty\], this implies that 
    \[\sup\frac{\mu(D_k)}{\whom(a_k)(1-a_k)}<\infty.\]
    \begin{center}
        \begin{tikzcd}[column sep=3cm,row sep=2cm]
    
            \awp(\Delta(a_k,4t)) \arrow[r, "T_k"] \arrow[d, "j_1"'] & L^q(D_k,\mu) \\
            H^\infty(\Delta(a_k,2t)) \arrow[r, "i"'] & \awp(\Delta(a_k,2t)) \arrow[u, "j_2"']
            \end{tikzcd} 
    \end{center}
    Hence, \(T_k\) is \(1-\)summing by \eqref{1tk}. 
    Let
    \begin{center}
        \begin{tikzcd}[column sep=3cm,row sep=2cm]

            \awp \arrow[r, "J_\mu"] \arrow[r, "J_\mu"] \arrow[d, "i_1"'] & L^q(\mu)\\
            \bigoplus\limits_{l^p}\awp(\Delta(a_k,4t)) \arrow[r, "T"'] & \bigoplus\limits_{l^p}L^q(D_k,\mu) \arrow[u, "i_2"']
            \end{tikzcd}
    \end{center}
    Then \(T\) is also \(1\)-summing, by  Lemma \ref{cor}, \[\pi_1(J_\mu)\les\sum_k\left(\frac{\mu(D_k)}{\whom(a_k)(1-a_k)}\right)^\frac{1}{p-1}. \]
\end{proof}

\section{Acknowledgement}
This research work was supported by the National Natural Science Foundation of China~(Grant No.~12261023,~11861023.)

\end{document}